\documentclass[reqno,11pt]{amsart}

\usepackage[letterpaper,margin=1in]{geometry}
\usepackage[T1]{fontenc}
\usepackage{amsmath,amssymb,mathtools}
\usepackage[colorlinks=true,linkcolor=blue,citecolor=blue,urlcolor=blue]{hyperref}
\usepackage[capitalize,nameinlink,noabbrev]{cleveref}
\usepackage{xcolor}

\newtheorem{theorem}{Theorem}
\newtheorem{lemma}{Lemma}
\newtheorem{proposition}{Proposition}
\newtheorem{corollary}{Corollary}
\theoremstyle{definition}
\newtheorem{definition}{Definition}
\newtheorem{remark}{Remark}

\newcommand{\CC}{\mathbb{C}}
\newcommand{\RR}{\mathbb{R}}
\newcommand{\T}{\mathbb{T}}
\newcommand{\cQ}{\mathcal{Q}}
\newcommand{\abs}[1]{\left\lvert #1\right\rvert}
\newcommand{\paren}[1]{\left(#1\right)}
\newcommand{\distT}{d_{\T}}

\title{Graphs with connectivity $3/4 - \varepsilon$ are globally synchronizing}

\author{Saba Lepsveridze}
\address{Department of Mathematics, Massachusetts Institute of Technology, Cambridge, MA 02139, USA}
\email{sabal@mit.edu}

\author{Sam Zhang}
\address{Department of Mathematics, Massachusetts Institute of Technology, Cambridge, MA 02139, USA}
\email{shzhang2@mit.edu}

\begin{document}

\begin{abstract}
We study synchronization in the Kuramoto model on finite graphs. We prove that there is an absolute constant $\eta>0$ such that every finite simple graph $G$ on $n$ vertices with minimum degree at least $(3/4-\eta)n$ has no local minima of the Kuramoto energy other than the fully synchronized states. This strictly improves the previous $3/4$ upper bound and refutes a conjecture of Bandeira, Kireeva, Maillard, and R\"odder.
\end{abstract}

\maketitle

\section{Introduction}

The spontaneous synchronization of coupled oscillators was famously observed by Christiaan Huygens in 1665, when two pendulum clocks suspended from a common support settled into a common period, with their pendulums moving in together \cite{bennett-huygens-clocks}. The Kuramoto model, introduced three centuries later, provides a canonical framework for studying how interactions produce collective synchronization \cite{kuramoto-75}. Since then, the Kuramoto model on general graphs and homogeneous networks has been extensively studied. In this general case, identical phase oscillators are placed at the vertices of a graph and coupled along its edges.

Let \(G=(V,E)\) be a finite simple graph on \(n\) vertices, with adjacency matrix \((w_{uv})_{u,v\in V}\). A phase configuration assigns a phase \(\theta_v\in\T:=\RR/2\pi\mathbb Z\) to each vertex. The associated Hamiltonian is
\begin{equation}\label{eq:energy}
    \mathcal E_G(\theta)
    :=\frac12\sum_{u,v\in V}w_{uv}
      \bigl(1-\cos(\theta_u-\theta_v)\bigr).
\end{equation}
For every nonempty connected graph, the global minimizers of \(\mathcal E_G\) are precisely the fully synchronized states, for which $\theta_u=\theta_v$ for all $u,v\in V$. We say that \(G\) is \emph{globally synchronizing} if these are its only local minima.\footnote{Some authors instead call \(G\) globally synchronizing when the negative gradient flow of \(\mathcal E_G\) converges to full synchrony from almost every initial condition. Our result does not immediately apply in that setting, but we strongly believe that it can be extended to apply there. See \cref{rem:final-remark} for more details.}

A central extremal problem asks how large the minimum degree of \(G\) must be to guarantee global synchronization, regardless of the remaining graph structure. We write \(\mu_{\mathrm c}\) for the corresponding critical connectivity. More concretely, define the connectivity\footnote{The $+1$ corresponds to adding a self-loop at every vertex. Such loops do not change the energy but makes the analysis cleaner.} by
\begin{equation*}
    \mu(G) := \frac{\min_{v \in V} \deg v + 1}{|V|}.
\end{equation*}
Then, the critical connectivity is defined by
\begin{equation*}
    \mu_{\mathrm c}
    := \inf \bigl\{\mu\in[0,1] : \text{every finite graph }G
    \text{ with }\mu(G)\geq\mu\text{ is globally synchronizing}\bigr\}.
\end{equation*}

It is not a priori clear that $\mu_{\mathrm c}<1$. Taylor established the first nontrivial upper bound,
\(\mu_{\mathrm c}\leq 0.9395\) \cite{taylor-dense}. Ling, Xu, and Bandeira improved this to \(0.7929\) \cite{ling-xu-bandeira}, and Lu and Steinerberger subsequently obtained \(0.7889\) \cite{lu-steinerberger}. Kassabov, Strogatz, and Townsend reached the substantially stronger bound $ \mu_{\mathrm c}\le 3/4$ in \cite{kassabov-strogatz-townsend-dense}.

At the same time, there are also constructions of graphs with large minimum degree that are not globally synchronizing. Townsend, Stillman, and Strogatz obtained a lower bound of \(0.6828\) from circulant constructions carrying stable nontrivial phase configurations \cite{townsend-stillman-strogatz}. Yoneda, Tatsukawa, and Teramae optimized the circulant twisted-state construction and increased this bound to \(0.6838\) \cite{yoneda-tatsukawa-teramae}. Finally, Canale constructed non-circulant graphs with strict nontrivial local minima and normalized minimum degree tending to \(11/16\) \cite{canale-spinning}. Consequently, before the present work the best rigorous bounds were
\begin{equation*}
    \frac{11}{16}\le \mu_{\mathrm c}\le\frac34.
\end{equation*}

The significance of the upper endpoint is not merely numerical. The argument of Kassabov, Strogatz, and Townsend identified the configuration responsible for the apparent barrier at \(3/4\). They constructed a sequence of graphs with connectivity approaching $3/4$ for which $\mathcal E_G(\theta)$ has nonsynchronized second-order critical points, although these points are unstable. Thus, \(3/4\) is the limit of their purely linear stability analysis, but their construction does not provide a genuine lower bound at that density.

This degeneracy motivated the conjecture that the upper bound \(3/4\) is sharp \cite[Conjecture~5]{randomstrasse-25}. Our main result refutes this conjecture and gives the first strict improvement over the $3/4$ upper bound.  

\begin{theorem}\label{thm:main}
There is an absolute constant \(\eta>0\) such that every finite simple graph \(G\) on \(n\) vertices satisfying
\[
    \min_{v \in V} \deg(v)\ge\paren{\frac34-\eta}n
\]
is globally synchronizing.
\end{theorem}

Our refined analysis shows that the configurations identified by Kassabov, Strogatz, and Townsend are essentially the only obstruction near \(3/4\). We then rule out this obstruction, yielding the strict improvement. Although an explicit value for  $\eta$ can be derived from our proof, we do not pursue it here.

\subsection{Proof Outline}

To explain the idea of the proof, suppose for contradiction that there is a sequence of growing graphs $G= G_n$ with connectivity $\mu(G)=3/4-o(1)$, together with nontrivial local minima $\theta = \theta^{(n)}$ of the Hamiltonian $\mathcal E_{G}$. The proof proceeds in two steps.

\begin{enumerate}
    \item The first step is to establish the following robustness result: after a global phase rotation, the phases $(\theta_v)_{v \in V}$ cluster into four groups lying within $o(1)$ of the phases $k\pi/2$, for $k\in\{0,1,2,3\}$, apart from an exceptional set of  $o(n)$ vertices; each of these clusters contain $(1/4+o(1))n$ vertices. Moreover, there are only $o(n^2)$ edges present between opposite clusters and only $o(n^2)$ edges missing within clusters or between adjacent clusters. 

    \item Next, we show that no such configuration can satisfy the
    second-order stability condition, except in a very special case with
    no exceptional vertices and no phase deviations within the four clusters. We prove this by testing the Hessian of the energy against a carefully chosen function that treats the vertices in each cluster in a unified manner while assigning a separate test value to each  exceptional vertex. This rules out every configuration except the exact four cluster case. The remaining configuration satisfies the  second-order condition, but an explicit perturbation changes the energy at cubic order and shows that it is not a local minimum.
\end{enumerate}

In this way, we show that every second-order local minimum is fully
synchronized, apart from the exact four cluster configuration, which is
precisely the barrier identified in
\cite[Section~6]{kassabov-strogatz-townsend-dense}. The cubic variation
then rules out this remaining configuration as a local minimum.

We emphasize that this clustering phenomenon was already suggested
heuristically in \cite[Section~6]{kassabov-strogatz-townsend-dense}. In
\cref{sec:order-parameters}, we give a self-contained, rigorous version of
the technical estimates underlying their $3/4$ upper-bound
argument. We then use these estimates to prove the clustering statement by a combinatorial
argument in \cref{sec:four-clusters}. Next, we construct  the special test function on the bulk in \cref{sec:bulk}, and \cref{sec:exceptional} extends it to the exceptional vertices. Finally, \cref{sec:exact-four-clusters}
rules out the exact four cluster case.

\subsection{Related Work}

Alongside the minimum degree problem discussed above, a parallel
line of work studies random networks and expander graphs. Ling, Xu, and
Bandeira \cite{ling-xu-bandeira} first proved that $G(n,p)$ is globally synchronizing with high
probability for $p=Cn^{-1/3}\log n$ with a sufficiently large constant
$C$. Kassabov, Strogatz, and Townsend \cite{kassabov-strogatz-townsend-random} subsequently
improved the sufficient range to $p\gg (\log n)^2/n$. Finally, Abdalla, Bandeira, Kassabov,
Souza, Strogatz, and Townsend \cite{bandeira-expanders} reached the optimal connectivity threshold\footnote{This is the connectivity threshold $\log n / n$ for Erd\H{o}s--R\'{e}nyi random graphs. It should not be confused with critical connectivity threshold $\mu_c$ defined above.}: for every fixed $\varepsilon>0$, they proved that $G(n,p)$ is globally
synchronizing with high probability whenever
$p\geq (1+\varepsilon)\log n/n$. In fact, Jain, Mizgerd, and Sawhney \cite{jain-mizgerd-sawhney} sharpened this further, proving that when edges get added randomly and uniformly to an initially independent set, the graph becomes globally synchronizing as soon it becomes connected. Thus connectivity is asymptotically the only obstruction in this model.

Note that \cite{bandeira-expanders} also present an expansion criteria. They show that if $G$ is $(n, d, \alpha d)$-expander graph, then it is globally synchronizing for $\alpha\leq 0.0816$. As consequences, they established global
synchronization for $d$-regular Ramanujan graphs and, with high probability, uniformly random $d$-regular graphs when $d\geq 600$. 

McRae \cite{mcrae-benign-landscapes} subsequently gave another deterministic criteria based the condition number of the Laplacian matrix. More precisely, for a connected graph to be  global synchronizing, it suffices that either the ordinary Laplacian $L$ or the normalized Laplacian $\mathcal{L}$ satisfy
\[
    \frac{\lambda_n(L)}{\lambda_2(L)}<2
    \qquad\text{or}\qquad
    \frac{\lambda_n(\mathcal{L})}{\lambda_2(\mathcal{L})}<2.
\]
For regular expanders, this improves the
allowable range to $\alpha<1/3$ and lowers the corresponding degree bound for both Ramanujan graphs and random $d$-regular graphs from $600$ to $35$. The same criterion gives short alternative proofs of the $3/4$ minimum degree theorem and of the random graph process result, as well as a new optimal result for a randomly signed Erd\H{o}s--R\'enyi model. 

Recent work has also uncovered structural mechanisms for synchronization
that do not rely on density or expansion. Wu and Brandes proved that all
connected threshold graphs are second-order globally synchronizing
\cite{wu-brandes-threshold}; they subsequently extended this conclusion
to all connected quasi-threshold graphs by showing how local
synchronization propagates along an underlying graph skeleton
\cite{wu-brandes-skeleton}. These results further illustrate that, outside
the worst case minimum degree setting considered here, the detailed graph structure can be more decisive than edge density alone.

\section*{Acknowledgments}

This research was completed at the SPUR program at Massachusetts Institute of Technology. We are grateful to David Jerison, Jonathan Bloom, and Oriol Sol\'e Pi for their insightful discussions and feedback, as well as Nike Sun for suggesting the problem. ChatGPT-5.6 Sol was used in editing and proofreading, as well as refining technical proofs. The proof strategy and any errors are entirely our own.

\section{Order parameter estimates}\label{sec:order-parameters}
In this section, following the argument of
\cite{kassabov-strogatz-townsend-dense}, we derive the two estimates that
drive the clustering analysis in \cref{sec:four-clusters}.  The main deliverables are 
\cref{lem:global-complement-c-sum-ub} and \cref{lem:rho1-rho2-ub-main}.

Note that every local minimum of $\mathcal{E}_G$ is a second-order critical point:
$\nabla\mathcal E_G(\theta)=0$ and
$\nabla^2\mathcal E_G(\theta)\succeq0$.  The first condition is equivalent
to the equilibrium equations
\begin{equation}\label{eq:first-order-condition-statement}
    \sum_{v\in V}w_{uv}\sin(\theta_u-\theta_v)=0
    \qquad (u\in V).
\end{equation}
The second is equivalent to
\begin{equation}\label{eq:second-order-condition-statement}
    \sum_{u,v\in V}w_{uv}\cos(\theta_u-\theta_v)
        \abs{F(u)-F(v)}^2\ge 0,
\end{equation}
for every $F\colon V\to\CC$. For a configuration $\theta$, define its first two order parameters by
\begin{equation}\label{eq:order-parameters}
    \rho_1:=\frac1n\sum_{v\in V}e^{i\theta_v},
    \qquad
    \rho_2:=\frac1n\sum_{v\in V}e^{2i\theta_v}.
\end{equation}
Thus $\rho_1$ is the average of the phase vectors.  If
$\abs{\rho_1}$ is close to $1$, then most phases are nearly aligned;
if it is small, the phase vectors exhibit substantial cancellation.
The second parameter $\rho_2$ measures the degree of antipodal coherence in the configuration.

In the rest of the paper, we call a local minimum $\theta$ \emph{nontrivial} if its phases are not
all equal. For notational brevity, set
\[
    C(x):=\cos x-\cos^2x.
\]
The following result appeared in \cite[Equation~6]{kassabov-strogatz-townsend-dense}; we give a self-contained proof for completeness.
\begin{lemma}\label{lem:global-complement-c-sum-ub}
If $(G,\theta)$ is a local minimum, then
\[
    \sum_{u,v\in V}(1-w_{uv})C(\theta_u-\theta_v)
    \le \frac{n^2}{2}
       \bigl(-1+2\abs{\rho_1}^2-\abs{\rho_2}^2\bigr).
\]
\end{lemma}

\begin{proof}
Apply \eqref{eq:second-order-condition-statement} with $F(v)=e^{i\theta_v}$. Since $ \abs{F(u) - F(v)}^2
    =2\bigl(1-\cos(\theta_u-\theta_v)\bigr),$ we obtain
\begin{equation}\label{eq:global-cos-cos2-stability}
    \sum_{u,v\in V}w_{uv}C(\theta_u-\theta_v)\ge0.
\end{equation}
On the other hand, since $\cos(\theta_u - \theta_v) = \Re \{F(u)\overline F(v)\},$ we get
\[
\sum_{u,v\in V}\cos(\theta_u-\theta_v)
    =\Bigl|\sum_{u \in V} F(u)\Bigr|^2
    =n^2\abs{\rho_1}^2,
\]
and similarly since $\cos(2\theta_u - 2\theta_v) = 2\cos^2(\theta_u - \theta_v) - 1$, we get
\[
\sum_{u,v\in V}\cos^2(\theta_u-\theta_v)
    =\frac{n^2}{2}
      +\frac12\Bigl|\sum_{v\in V}F(v)^2\Bigr|^2
    =\frac{n^2}{2}\bigl(1+\abs{\rho_2}^2\bigr).
\]
Consequently, 
\begin{equation}\label{eq:global-cos-cos2-sum}
\sum_{u,v\in V}C(\theta_u-\theta_v)
  =\frac{n^2}{2}
    \bigl(-1+2\abs{\rho_1}^2-\abs{\rho_2}^2\bigr).
\end{equation}
Subtracting \eqref{eq:global-cos-cos2-stability} from
\eqref{eq:global-cos-cos2-sum} proves the claim.
\end{proof}

The preceding estimate isolates the contribution of the complement of
$G$.  We next use the minimum degree hypothesis and the equilibrium equations to obtain a lower bound for the same complement sum.  Comparing the two bounds yields a quantitative constraint on $\rho_1$ and $\rho_2$. The proof strategy is a modification of \cite[Lemma~1]{kassabov-strogatz-townsend-dense}.

\begin{proposition}\label{prop:main-rho1-rho2-bound}
Suppose that $\theta$ is a local minimum of $\mathcal E_G$ and that
$\mu(G)\ge\mu$, where $\mu\in[0,1]$.  Then
\[
1-2\abs{\rho_1}^2+\abs{\rho_2}^2
\le \sqrt{16(1-\mu)^2
       -8\abs{\rho_1}^2
        \bigl(1-\abs{\rho_2}\bigr)}.
\]
\end{proposition}

\begin{proof}
Rotate all phases by a common angle so that
$\rho_1=\abs{\rho_1}\ge0$ is real.  Adding a self-loop at every vertex changes neither
the energy nor the first and second order conditions, so throughout the
proof we take $w_{uu}=1$.  Using
\eqref{eq:first-order-condition-statement} and the fact that $\rho_1$
is real, we have
\begin{align*}
\sum_{v\in V}(1- w_{uv})\sin(\theta_u-\theta_v)=\sum_{v\in V}\sin(\theta_u-\theta_v) =n \rho_1 \sin\theta_u.
\end{align*}
Applying Cauchy--Schwarz, using that $w_{uv} \in \{0, 1\}$, gives
\begin{equation*}
    n^2\rho_1^2\sin^2\theta_u\le \Bigl[\sum_{v\in V}(1- w_{uv})\Bigr]\Bigl[\sum_{v\in V} (1- w_{uv}) \sin^2(\theta_u-\theta_v)\Bigr].
\end{equation*}
Another Cauchy--Schwarz now gives 
\begin{align*}
\Bigl[\sum_v(1- w_{uv})
       \abs{\cos(\theta_u-\theta_v)}\Bigr]^2
&\le \Bigl[\sum_{v\in V}(1- w_{uv})\Bigr]\Bigl[\sum_{v\in V} (1-w_{uv}) \cos^2(\theta_u-\theta_v)\Bigr]\\
&=\Bigl[\sum_{v\in V}(1- w_{uv})\Bigr]^2-\Bigl[\sum_{v\in V}(1- w_{uv})\Bigr]\Bigl[\sum_v (1- w_{uv})
       \sin^2(\theta_u-\theta_v)\Bigr]\\
&\le n^2\bigl[(1-\mu)^2-\rho_1^2\sin^2\theta_u\bigr].
\end{align*}
In particular, every expression under a square root below is
nonnegative.  Since $t-t^2\ge-2\abs{t}$ for any $t\in[-1,1]$, it follows that
\[
\sum_v (1-w_{uv}) C(\theta_u-\theta_v)
\ge-2n\sqrt{(1-\mu)^2-\rho_1^2\sin^2\theta_u}.
\]
After summing over $u$, Jensen's inequality yields
\begin{equation}\label{eq:complement-c-lower-bound-pre-rho2}
\sum_{u,v}(1- w_{uv})C(\theta_u-\theta_v)
\ge-2n^2\sqrt{(1-\mu)^2
       -\frac{\rho_1^2}{n}\sum_u\sin^2\theta_u}.
\end{equation}
Furthermore, note that 
\[
    \frac1n\sum_u\sin^2\theta_u
    =\frac12\bigl(1-\Re\rho_2\bigr)
    \ge\frac12(1-|\rho_2|).
\]
Thus \eqref{eq:complement-c-lower-bound-pre-rho2} implies
\[
\sum_{u,v}(1-w_{uv})C(\theta_u-\theta_v)
\ge-2n^2\sqrt{(1-\mu)^2-\frac{\rho_1^2}{2}(1-|\rho_2|)}.
\]
On the other hand, \cref{lem:global-complement-c-sum-ub} gives
\[
\sum_{u,v}(1- w_{uv})C(\theta_u-\theta_v)
\le-\frac{n^2}{2}\bigl(1-2\rho_1^2+|\rho_2|^2\bigr).
\]
Comparing the last two inequalities proves the claim.
\end{proof}

\begin{lemma}[Half-circle lemma \cite{bandeira-expanders}]
\label{lem:half-circle-lemma}
Suppose $(G,\theta)$ is an equilibrium of a connected graph and, after a
global rotation, $\cos\theta_v>0$ for every $v\in V$.  Then $\theta$ is
fully synchronized.
\end{lemma}

\begin{proof}
Choose representatives with $-\pi/2<\theta_v<\pi/2$, and let $u$ maximize
$\theta_u$.  Since $0\le\theta_u-\theta_v<\pi$ for every $v$, each
summand in
\[
    \sum_vw_{uv}\sin(\theta_u-\theta_v)=0
\]
is nonnegative, so every neighbor of $u$ also has the maximal phase.
The set of vertices with maximal phase therefore has no edge to its
complement.  Connectedness implies that this set is all of $V$.
\end{proof}


\begin{lemma}\label{lem:rho1-rho2-ub-main}
For every $\varepsilon>0$, there is an $\eta>0$ such that every nontrivial
local minimum $(G,\theta)$ with $\mu(G)\ge3/4-\eta$ satisfies
\[
    \abs{\rho_1}\le\varepsilon,
    \qquad
    \abs{\rho_2}\le\varepsilon.
\]
\end{lemma}

\begin{proof}
Let $\theta$ be a nontrivial local minimum on a connected graph, and rotate the phases so that $\rho_1=\abs{\rho_1}\ge0$ is real.  As in the proof
of \cref{prop:main-rho1-rho2-bound}, add a self-loop at every vertex, so
$w_{uu}=1$.  For $u\in V$, set
\[
    r_u:=\frac1n\sum_vw_{uv}\cos(\theta_u-\theta_v).
\]
The equilibrium equations give
\[
    \frac1n\sum_vw_{uv}e^{i\theta_v}=r_u e^{i\theta_u}
\]
with $r_u$ a real number. Testing \eqref{eq:second-order-condition-statement} on the indicator of
$u$ shows that the contribution from the non-loop edges incident to $u$
is nonnegative; hence $r_u > 0$.  It follows that
\begin{equation}\label{eq:rho1-local-disk-bound}
    \Bigl|\rho_1-r_u e^{i\theta_u}  \Bigr|
    =\frac1n\Bigl|\sum_v(1-w_{uv})e^{i\theta_v}\Bigr|
    \le1-\mu(G).
\end{equation}
If $\rho_1>1-\mu(G)$, then \eqref{eq:rho1-local-disk-bound} implies $\Re(r_u e^{i\theta_u})>0$ for every $u$.  Since
$r_u>0$, all phases satisfy $\cos\theta_u>0$, and
\cref{lem:half-circle-lemma} makes the configuration fully synchronized, which is a contradiction.  We have therefore proved that every such nontrivial
local minimum satisfies
\begin{equation}\label{eq:rho1-connectivity-bound}
    \abs{\rho_1}\le1-\mu(G).
\end{equation}

Now suppose the lemma fails for some $\varepsilon>0$.  Then there are numbers $\eta_k\downarrow0$ and nontrivial local minima
$(G_k,\theta^{(k)})$ such that
$\mu(G_k)\ge3/4-\eta_k$ and at least one of the asserted bounds fails.
For all sufficiently large $k$, we have $\mu(G_k)>1/2$, which forces
$G_k$ to be connected.  Rotate each configuration so that
\[
    x_k:=\rho_1(\theta^{(k)})
        =\bigl|\rho_1(\theta^{(k)})\bigr|\ge0,
    \qquad
    y_k:=\bigl|\rho_2(\theta^{(k)})\bigr|.
\]
By \eqref{eq:rho1-connectivity-bound}, we have
\[
    x_k\le1-\mu(G_k)\le\frac14+\eta_k.
\]
Since $0\le x_k,y_k\le1$, pass to a subsequence for which
$x_k\to x$ and $y_k\to y$.  The failure of the asserted bounds gives
$\max\{x,y\}\ge\varepsilon$, while $0\le x\le1/4$. Apply \cref{prop:main-rho1-rho2-bound} with
$\mu=3/4-\eta_k$.  Since
\[
    1-2x_k^2+y_k^2
    \ge1-2\Bigl(\frac14+\eta_k\Bigr)^2>0
\]
for all sufficiently large $k$, we may square the resulting inequality
and pass to the limit.  This gives
\begin{equation}\label{eq:limiting-order-parameter-polynomial}
0\le x^2(8y-4)+y^2(4x^2-2)-4x^4-y^4.
\end{equation}
For $0\le y\le1$, the inequality $8y\le2+8y^2$ and the bound
$x^2\le1/16$ imply that the right-hand side of
\eqref{eq:limiting-order-parameter-polynomial} is at most
\begin{align*}
    -2x^2+(12x^2-2)y^2-4x^4-y^4
    &\le-2x^2-\frac54y^2-4x^4-y^4.
\end{align*}
It follows that $x=y=0$, contradicting
$\max\{x,y\}\ge\varepsilon$.
\end{proof}

\section{Four-cluster structure}\label{sec:four-clusters}

We now show that small first and second order parameters force four large
clusters near quarter turns.  In particular, we will use \cref{lem:global-complement-c-sum-ub} and  
\cref{lem:rho1-rho2-ub-main} to prove clustering. All distances between phases are measured by
the circular distance $\distT$ on $\T$.

\begin{lemma}\label{lem:clustering-part-1}
For every $\varepsilon>0$, there is an $\eta>0$ such that the following
holds.  If $\mu(G)\ge3/4-\eta$ and a local minimum $\theta$ satisfies
$\abs{\rho_1},\abs{\rho_2}<\eta$, then there are a phase
$\theta_*\in\T$ and four disjoint sets $S_j\subset V(G)$ with $0\le j \le 3$  such that
\[
    \abs{S_j}\ge\Bigl({\frac14-\varepsilon}\Bigr)n,
    \qquad
    \distT\Bigl({\theta_v,\theta_*+\frac{\pi j}{2}\Bigr)}
    \le\varepsilon
\]
for every $v\in S_j$ and $0\le j\le3$.
\end{lemma}

\begin{proof}
Suppose not.  For some fixed $\varepsilon>0$, there is a sequence of
numbers $\eta_k\downarrow0$ and corresponding counterexamples.  We suppress the index $k$ and use $o(1)$ for quantities tending to zero along
this sequence. Let $\overline G$ be the complement of $G$, and put
$m=|{E(\overline G)}|$.  The minimum degree hypothesis gives
\[
    m\le\Bigl(\frac18+o(1)\Bigr)n^2.
\]
On the other hand, dividing the ordered sum \eqref{eq:global-cos-cos2-sum} in
\cref{lem:global-complement-c-sum-ub} by two and using $C(\theta)\ge-2$ give
\[
-\Bigl({\frac14-o(1)}\Bigr)n^2 \geq  \sum_{uv\in E(\overline G)}C(\theta_u-\theta_v) \, \geq -2m
\]
Consequently, we get that 
\begin{equation}\label{eq:complement-near-equality}
    m=\Bigl({\frac18+o(1)}\Bigr)n^2,
    \qquad
    \sum_{uv\in E(\overline G)}
       \bigl[C(\theta_u-\theta_v)+2\bigr]=o(n^2).
\end{equation}

The function $C+2$ is nonnegative and vanishes only at $\pi$ modulo $2\pi$.  Let $a$ be its average over $E(\overline G)$, so $a=o(1)$ by
\eqref{eq:complement-near-equality}.  Retain those complement edges for
which $C+2\le\sqrt a$, and call the
resulting spanning subgraph $H$. By Markov's inequality, the number of discarded edges are most
$\sqrt a\,m=o(n^2)$.  The uniqueness of the zero of $C+2$ then gives, uniformly over the retained edges,
\begin{equation}\label{eq:h-asymptotic-properties}
    \abs{E(H)}=\Bigl({\frac18+o(1)}\Bigr)n^2,
    \qquad
    \distT(\theta_v,\theta_u+\pi)=o(1)
    \;\; \text{ for all } uv\in E(H).
\end{equation}

Choose $r_k\downarrow0$ slowly enough that simultaneously
\[
    \max \deg (H)\le\Bigl({\frac14+r_k}\Bigr)n,
    \qquad
    \frac{2\abs{E(H)}}n\ge\Bigl({\frac14-r_k}\Bigr)n.
\]
Comparing the total degree with the upper bound, Markov's inequality shows that all but at most $O(\sqrt{r_k}\,n)=o(n)$ vertices have $H$-degree at least $(1/4-\sqrt{r_k})n$; call these vertices \emph{good}.  Choose a good vertex $u$.
It has a good neighbor $v$, since it has $(1/4+o(1))n$ neighbors and only
$o(n)$ vertices are not good. Then the sets 
\[
    S_0:=N_H(v),\qquad S_2:=N_H(u)
\]
each have size $(1/4+o(1))n$ and lie within $o(1)$ of $\theta_u$ and
$\theta_u+\pi$, respectively.  Moreover, they are disjoint as  a
common vertex would be nearly antipodal to both endpoints of the edge $uv$, even though those endpoints are themselves nearly antipodal.

Let $R=V\setminus(S_0\cup S_2)$ and set $\theta_*=\theta_u$.  Then
$\abs R=(1/2+o(1))n$.  Since $\abs{\rho_2}=o(1)$ and every vertex $x \in S_0\cup S_2$ satisfies $\cos(2[\theta_x-\theta_*]) = 1-o(1)$, we have
\begin{equation}\label{eq:remaining-doubled-phases}
    \sum_{x\in R}
       \bigl[1+\cos(2[\theta_x-\theta_*])\bigr]=o(n).
\end{equation}
Applying Markov's inequality again with a threshold tending sufficiently slowly to zero shows that all but $o(n)$ vertices of
$R$ lie uniformly within $o(1)$ of one of $\theta_*+\pi/2$ and
$\theta_*+3\pi/2$.  Choose such a vertex $u'$ that is good, and then choose
a good neighbor $v'$ of $u'$ in $H$.  After interchanging the two
quarter-turn centers if necessary, the sets
\[
    S_1:=N_H(v'),\qquad S_3:=N_H(u')
\]
have size $(1/4+o(1))n$ and lie within $o(1)$ of
$\theta_*+\pi/2$ and $\theta_*+3\pi/2$, respectively.  The four sets are
disjoint for large indices because they lie in four arcs of radius $o(1)$
whose centers are separated by exactly $\pi/2$.  They then satisfy the conclusion
with the fixed $\varepsilon$, which is a contradiction.
\end{proof}

In the following definition, we establish the terminology and notation for clustering used throughout the rest of the paper. It is worth noting that the parameters $h, \rho, \alpha,$ and $\beta$ will be taken to be small in the future arguments.

\begin{definition}\label{def:clustering-definition}
Fix $h\in[0,\pi/4)$, and $\rho, \alpha,\beta\geq 0$. Let $G = (V,E)$ be a graph and $\theta : V \to \mathbb{T}$ be a phase configuration. For $j\in\{0,1,2,3\}$, set
\[
    B_j:=\{v\in V:\distT(\theta_v,\pi j/2)\le h\},
    \qquad
    B:=\bigsqcup_{j=0}^3B_j,
    \qquad
    E:=V\setminus B.
\]
 We call $B$ the \emph{bulk} and the vertices in $E$ \emph{exceptional}.  For $v\in B_j$, let
$x_v\in[-h,h]$ be the offset so that $\theta_v=\pi j/2+x_v$ in $\T$. We further define
\[
    \bar t:=\frac1{\abs B}\sum_{v\in B}x_v,
    \qquad
    t_v:=x_v-\bar t.
\]
We say that $(G,\theta)$ is
\emph{$(h,\rho,\alpha,\beta)$-clustered} if the following hold:
\begin{enumerate}
\item $(1/4-\rho)n\le\abs{B_j}\le(1/4+\rho)n$ for every $j$;
\item $\sum_{v\in B}\abs{t_v}\le\alpha |B|$ and $|\bar t|\le\alpha$;
\item $\abs E\leq\beta n$.
\end{enumerate}
\end{definition}

We are now finally ready to prove the main clustering theorem, which will be central to the subsequent sections. To this end we will combine \cref{lem:rho1-rho2-ub-main} and \cref{lem:clustering-part-1}. 

\begin{theorem}\label{thm:clustering-main}
Fix $h\in(0,\pi/4)$, and $\rho, \alpha, \beta > 0$. There is an $\eta>0$ such that every nontrivial
local minimum $(G,\theta)$ with
$\mu(G)\ge3/4-\eta$ is $(h,\rho,\alpha,\beta)$-clustered after a global rotation.
\end{theorem}

\begin{proof}
Let us fix $\varepsilon>0$ sufficiently small so that all subsequent estimates hold. We first apply \cref{lem:clustering-part-1} with $\varepsilon$, and denote
its minimum degree parameter by $\gamma>0$.  Next apply
\cref{lem:rho1-rho2-ub-main} with $\gamma/2$, choosing the
resulting parameter $\eta>0$ so that $\eta\le\gamma$.

If $(G,\theta)$ is a nontrivial local minimum with
$\mu(G)\ge3/4-\eta$, then \cref{lem:rho1-rho2-ub-main} gives
$|{\rho_1}|,|{\rho_2}|\le\gamma/2$.  Hence
\cref{lem:clustering-part-1} shows that there is a phase $\theta_*$ and disjoint
sets $S_0,S_1,S_2,S_3$ such that
\[
    \abs{S_j}\ge\Bigl({\frac14-\varepsilon}\Bigr)n,
    \qquad
    \distT(\theta_v,\theta_*+\pi j/2)\le\varepsilon
    \quad(v\in S_j).
\]

After a global rotation, assume $\theta_*=0$, and form $B_j,B,E$ as in
\cref{def:clustering-definition}.  Then, as long as $\varepsilon$ is sufficiently small, we have that $S_j\subseteq B_j$, while
$S_k\cap B_j=\varnothing$ for $k\ne j$.  Therefore
\[
    \Bigl({\frac14-\rho}\Bigr)n
    \le\Bigl({\frac14-\varepsilon}\Bigr)n
    \le\abs{S_j}
    \le\abs{B_j} \le n-\sum_{k\ne j}\abs{S_k}
    \le\Bigl({\frac14+3\varepsilon}\Bigr)n
    \le\Bigl({\frac14+\rho}\Bigr)n.
\]
This shows property (1). We also have $|E| = n - |B| \leq 4\varepsilon n \leq \beta n$, which shows property (3). Finally, since $\bigcup_jS_j\subseteq B$, at most $4\varepsilon n$ vertices of $B$ have $\abs{x_v}>\varepsilon$. Hence,
\begin{equation*}
    \frac{1}{|B|}\sum_{v \in B}|x_v| \leq \varepsilon + \frac{4\varepsilon h}{1-4\varepsilon} \leq \frac{\alpha}{2}.
\end{equation*}
Thus $\abs{\bar t}\le \alpha/2$, which in turn shows property (2) as 
\[
    \frac{1}{|B|}\sum_{v\in B}\abs{t_v}
    \le \abs{\bar t} +  \frac{1}{|B|}\sum_{v \in B}|x_v|  \leq \alpha.
\]
This concludes the proof of clustering.
\end{proof}

\section{Test function on the bulk}\label{sec:bulk}
In this section and the next, we take the clustering result from \cref{thm:clustering-main} as an input. That is, we will be assuming that $(G,\theta)$ is $(h,\rho,\alpha,\beta)$-clustered nontrivial local minimum.  We construct a test function $F : V \to \CC$ to certify that the Hessian is not positive semidefinite.  To this end, we first define $F$ on the bulk. Define the restriction $F\vert_B\colon B\to\CC$ by\footnote{The motivation for this choice of test function is to approximate, to first order, the test function $F'(v) = e^{i\theta_v}$ used in \cref{lem:global-complement-c-sum-ub}.}
\begin{equation}\label{eq:test-function-bulk}
    F\vert_B(v):=i^j(1+it_v)\qquad(v\in B_j).
\end{equation}

In the rest of the section, we will upper bound the bulk contribution to \eqref{eq:second-order-condition-statement}, culminating in the main result of this section in \cref{cor:bulk-bound-main}. 

\medskip

For $S_1,S_2\subseteq V$ and $H\colon S_1\cup S_2\to\CC$, define
\[
\cQ_{S_1\times S_2}(H)
:=\sum_{u\in S_1}\sum_{v\in S_2}
 w_{uv}\cos(\theta_u-\theta_v)\abs{H(u)-H(v)}^2.
\]
The second-order condition \eqref{eq:second-order-condition-statement} says $\cQ_{V\times V}(F)\ge 0$. 
Throughout the argument, $C$ will denote a sufficiently large absolute
constant.  It may be enlarged from line to line but the final choice is understood in the statements of the lemmas.  We assume that $h>0$ is smaller than some fixed absolute constant; its value will be chosen after the bulk estimate is established.  All remaining error parameters will then be chosen sufficiently small in terms of $h$.

For a fixed clustering, abbreviate
\begin{align*}
\sum_{\mathrm{same}}f(u,v)
&:=\sum_{j=0}^3\sum_{u\in B_j}\sum_{v\in B_j}f(u,v),\\
\sum_{\mathrm{adj}}f(u,v)
&:=\sum_{j=0}^3\sum_{u\in B_j}\sum_{v\in B_{j+1}}f(u,v),\\
\sum_{\mathrm{opp}}f(u,v)
&:=\sum_{j=0}^3\sum_{u\in B_j}\sum_{v\in B_{j+2}}f(u,v).
\end{align*}
Thus the same-cluster and opposite-cluster sums count undirected edges twice, whereas the adjacent-cluster sum counts them once.  We also write $\Delta_{uv}:=t_u-t_v,$ and note that it is typically tiny.

We first use the first-order condition \eqref{eq:first-order-condition-statement} to capture the cancellation between adjacent clusters. This cancellation will be crucial for bounding $\cQ_{B \times B}(F)$ in \cref{lem:q-b-b-bound-step-1}. 

To see the intuition behind \cref{lem:first-order-condition}, fix \(u\in B_j\). For vertices in \(B_j\) or \(B_{j+2}\), the phase difference is close to a multiple of \(\pi\), so the corresponding sine term is small. By contrast, the phase differences from \(u\) to \(B_{j-1}\) and \(B_{j+1}\) are close to \(-\pi/2\) and \(\pi/2\), respectively. The first-order condition therefore forces the contributions from these two adjacent clusters to nearly cancel. In precise terms, although the sum
\[
\sum_{\mathrm{adj}}w_{uv}\Delta_{uv}\cos\Delta_{uv}
\]
appears to be linear in \(\Delta_{uv}\), \cref{lem:first-order-condition} expresses it in terms of sums of smaller order. 

\begin{lemma}\label{lem:first-order-condition}
Suppose $(G,\theta)$ is $(h,\rho,\alpha,\beta)$-clustered.  There is a
number $\mathcal R$ with $\abs{\mathcal R}\le n^2\alpha\beta$ such that
\begin{equation}\label{eq:first-order-summed}
\sum_{\mathrm{adj}}w_{uv}\Delta_{uv}\cos\Delta_{uv}=
-\frac12\sum_{\mathrm{opp}}w_{uv}\Delta_{uv}\sin\Delta_{uv}+\frac12\sum_{\mathrm{same}}w_{uv}\Delta_{uv}\sin\Delta_{uv}
+\mathcal R.
\end{equation}
\end{lemma}

\begin{proof}
For $u\in B_j$, the first-order condition \eqref{eq:first-order-condition-statement} translates to
\begin{align*}
0=&\sum_{v\in B_j}w_{uv}\sin\Delta_{uv}
-\sum_{v\in B_{j+1}}w_{uv}\cos\Delta_{uv}
+\sum_{v\in B_{j-1}}w_{uv}\cos\Delta_{uv}\\
&-\sum_{v\in B_{j+2}}w_{uv}\sin\Delta_{uv}
+\sum_{v\in E}w_{uv}\sin(\theta_u-\theta_v).
\end{align*}
Multiply the whole equation by $t_u$ and sum over $u\in B$. Note that symmetrization gives 
\begin{gather*}
\sum_{j = 0}^3\sum_{u\in B_j}\sum_{v\in B_j}t_uw_{uv}\sin\Delta_{uv}
=\frac12\sum_{\mathrm{same}}
  w_{uv}\Delta_{uv}\sin\Delta_{uv} \\
    \sum_{j = 0}^3\sum_{u\in B_j}\sum_{v\in B_{j+2}}
t_uw_{uv}\sin\Delta_{uv}
=\frac12\sum_{\mathrm{opp}}
  w_{uv}\Delta_{uv}\sin\Delta_{uv}. \\ 
   \sum_{j = 0}^3\sum_{u\in B_j }\sum_{v\in B_{j-1}} t_uw_{uv}\cos\Delta_{uv} -  \sum_{j = 0}^3\sum_{u\in B_j }\sum_{v\in B_{j+1}} t_uw_{uv}\cos\Delta_{uv}
=-\sum_{\mathrm{adj}}w_{uv}\Delta_{uv}\cos\Delta_{uv}
\end{gather*}
Therefore, \eqref{eq:first-order-summed} holds with
\[
    \mathcal R=\sum_{u\in B}\sum_{v\in E}
       t_uw_{uv}\sin(\theta_u-\theta_v).
\]
The defining $L^1$ bound for the clustering gives
$\abs{\mathcal R}\le\abs E\sum_{u\in B}\abs{t_u}\le n^2\alpha\beta$.
\end{proof}

We now define $\sigma\ge0$ by
\begin{equation}\label{eq:sigma-definition}
    n\sigma^2:=\sum_{v\in B}t_v^2.
\end{equation}
Our eventual goal is to bound $\mathcal{Q}_{B\times B}(F)$, and later $\mathcal{Q}_{V\times V}(F)$, relative to $\sigma^2$, so that the second-order condition $\mathcal{Q}_{V \times V}(F) = 0$ will ultimately force $\sigma = 0$.

\begin{lemma}\label{lem:q-b-b-bound-step-1}
For the test function in \eqref{eq:test-function-bulk},
\begin{equation}\label{eq:q-b-times-b-final-ub}
\cQ_{B\times B}(F)
\le 3\sum_{\mathrm{same}}\Delta_{uv}^2
-4\sum_{\mathrm{adj}}w_{uv}\Delta_{uv}^2
-\paren{4-Ch^2}\sum_{\mathrm{opp}}w_{uv}\notag+Cn^2\paren{\alpha\beta+h\sigma^2}.
\end{equation}
\end{lemma}

\begin{proof}
A direct expansion gives
\begin{align}\label{eq:q-b-times-b-expansion}
\cQ_{B\times B}(F)
={}&\sum_{\mathrm{same}}w_{uv}\Delta_{uv}^2\cos\Delta_{uv}
+4\sum_{\mathrm{adj}}w_{uv}\sin\Delta_{uv}
-4\sum_{\mathrm{adj}}w_{uv}\Delta_{uv}\sin\Delta_{uv}\notag\\
&+2\sum_{\mathrm{adj}}w_{uv}(t_u^2+t_v^2)\sin\Delta_{uv}
-\sum_{\mathrm{opp}}w_{uv}
  \bigl(4+(t_u+t_v)^2\bigr)\cos\Delta_{uv}.
\end{align}
Since $\abs{t_v}\le h+\alpha\le2h$, we have
$\abs{\Delta_{uv}}\le4h$ and
\begin{equation}\label{eq:cubic-term-bound}
\sum_{u,v\in B}\abs{\Delta_{uv}}^3 \le4h\sum_{u,v\in B}\Delta_{uv}^2
 \le8h\sum_{u,v\in B}(t_u^2+t_v^2)\le Chn^2\sigma^2.
\end{equation}
We estimate the five terms in \eqref{eq:q-b-times-b-expansion} in order.
First,
\[
\sum_{\mathrm{same}}w_{uv}\Delta_{uv}^2\cos\Delta_{uv}
\le\sum_{\mathrm{same}}\Delta_{uv}^2.
\]
Next, the elementary inequality
$\sin x\le x\cos x+\abs{x}^3$ for $\abs{x}\le4h$, together with
\cref{lem:first-order-condition} and \eqref{eq:cubic-term-bound}, gives
\begin{align*}
4\sum_{\mathrm{adj}}w_{uv}\sin\Delta_{uv}
&\le4\sum_{\mathrm{adj}}w_{uv}\Delta_{uv}\cos\Delta_{uv}
  +4\sum_{\mathrm{adj}}\abs{\Delta_{uv}}^3\\
&\le2\sum_{\mathrm{same}}w_{uv}\Delta_{uv}\sin\Delta_{uv}
  -2\sum_{\mathrm{opp}}w_{uv}\Delta_{uv}\sin\Delta_{uv}
  +Cn^2\paren{\alpha\beta+h\sigma^2}\\
&\le2\sum_{\mathrm{same}}w_{uv}\Delta_{uv}^2
  +Cn^2\paren{\alpha\beta+h\sigma^2}.
\end{align*}
Here $x\sin x\ge0$ and $x\sin x\le x^2$ on the relevant interval.
Also,
\begin{equation*}
-4\sum_{\mathrm{adj}}w_{uv}\Delta_{uv}\sin\Delta_{uv}
\le-4\sum_{\mathrm{adj}}w_{uv}\Delta_{uv}^2
  +\sum_{\mathrm{adj}}w_{uv}\Delta_{uv}^4\le-4\sum_{\mathrm{adj}}w_{uv}\Delta_{uv}^2
  +Ch^2n^2\sigma^2
\end{equation*}
using that $x \sin x \ge x^2 - x^4/4$. Furthermore,
\[
2\sum_{\mathrm{adj}}w_{uv}(t_u^2+t_v^2)\sin\Delta_{uv}
\le Chn^2\sigma^2.
\]
Finally, $\cos\Delta_{uv}\ge1-8h^2$, and hence
\[
-\sum_{\mathrm{opp}}w_{uv}
\bigl(4+(t_u+t_v)^2\bigr)\cos\Delta_{uv}
\le\paren{-4+Ch^2}\sum_{\mathrm{opp}}w_{uv}.
\]
Adding these estimates proves \eqref{eq:q-b-times-b-final-ub}.
\end{proof}

\begin{lemma}\label{lem:q-b-b-bound-step-2}
Suppose $(G,\theta)$ is $(h,\rho,\alpha,\beta)$-clustered and
$\mu(G)\ge3/4-\eta$.  Then
\begin{equation}\label{eq:bulk-combinatorial-bound}
3\sum_{\mathrm{same}}\Delta_{uv}^2
-4\sum_{\mathrm{adj}}w_{uv}\Delta_{uv}^2
-\paren{4-Ch^2}\sum_{\mathrm{opp}}w_{uv}
\le\paren{-\frac12+C\rho+C\eta}n^2\sigma^2
\end{equation} for sufficiently small $h$.
\end{lemma}

\begin{proof}
Write
\begin{equation}\label{eq:delta-squared-sum-complement-rewrite}
3\sum_{\mathrm{same}}\Delta_{uv}^2
-4\sum_{\mathrm{adj}}w_{uv}\Delta_{uv}^2
=3\sum_{\mathrm{same}}\Delta_{uv}^2
-4\sum_{\mathrm{adj}}\Delta_{uv}^2
+4\sum_{\mathrm{adj}}(1-w_{uv})\Delta_{uv}^2.
\end{equation}
Let $T_j=\sum_{u\in B_j}t_u$.  Then
\begin{align*}
\sum_{\mathrm{same}}\Delta_{uv}^2
&\le\paren{\frac12+2\rho}n^2\sigma^2
  -2\sum_jT_j^2,\\
\sum_{\mathrm{adj}}\Delta_{uv}^2
&\ge\paren{\frac12-2\rho}n^2\sigma^2
  -2\sum_jT_jT_{j+1}.
\end{align*}
Consequently,
\begin{align}
3\sum_{\mathrm{same}}\Delta_{uv}^2
-4\sum_{\mathrm{adj}}\Delta_{uv}^2
&\le\paren{-\frac12+C\rho}n^2\sigma^2
 -6\sum_jT_j^2+8\sum_jT_jT_{j+1}\notag\\
&\le\paren{-\frac12+C\rho}n^2\sigma^2.
\label{eq:same-adj-delta-sum}
\end{align}
The last step uses $\sum_jT_j=0$ (from the choice of $\bar t$ and $t_v$ in \cref{def:clustering-definition}) and
\[
    \sum_jT_jT_{j+1}
    =(T_0+T_2)(T_1+T_3)=-(T_0+T_2)^2\le0.
\]
For the complement term, $\sum_v(1-w_{uv})\le(1/4+\eta)n$ and
$\Delta_{uv}^2\le2(t_u^2+t_v^2)$ give
\begin{align}
\sum_{\mathrm{adj}}(1-w_{uv})\Delta_{uv}^2
&\le2\sum_j\sum_{u\in B_j}t_u^2
 \sum_{v\in B_{j-1}\cup B_{j+1}}(1-w_{uv})\notag\\
&\le2\sum_j\sum_{u\in B_j}t_u^2
 \left((\eta+\rho)n+\sum_{v\in B_{j+2}}w_{uv}\right)\notag\\
&\le2(\eta+\rho)n^2\sigma^2
  +2(h+\alpha)^2\sum_{\mathrm{opp}}w_{uv}.
\label{eq:adj-complement-sum}
\end{align}
Combining \eqref{eq:delta-squared-sum-complement-rewrite},
\eqref{eq:same-adj-delta-sum}, and \eqref{eq:adj-complement-sum}, the
left-hand side of \eqref{eq:bulk-combinatorial-bound} is at most
\[
\paren{-\frac12+C\rho+C\eta}n^2\sigma^2
-\bigl(4-Ch^2-8(h+\alpha)^2\bigr)
 \sum_{\mathrm{opp}}w_{uv}.
\]
The last term is nonpositive once the $h$ is chosen sufficiently small, since $\alpha\le h/2$.
\end{proof}

\begin{corollary}\label{cor:bulk-bound-main}
There are absolute constants $h_0,c,C>0$ with the following property.
If $h \in (0, h_0)$, and  $(G,\theta)$ is
$(h,\rho,\alpha,\beta)$-clustered with $\mu(G)\ge3/4-\eta$, then
\begin{equation}\label{eq:bulk-bound-main}
\cQ_{B\times B}(F)
\le\paren{-c+C\rho+C\eta}n^2\sigma^2
  +Cn^2\alpha\beta.
\end{equation}
\end{corollary}

\begin{proof}
Combine \cref{lem:q-b-b-bound-step-1,lem:q-b-b-bound-step-2}, and choose
$h_0$ small enough to absorb the term $Ch$ into the negative constant
and to ensure $h_0<\pi/4$.
\end{proof}

For the remainder of the proof, fix one $h\in(0,h_0]$.

\section{Test function outside the bulk}\label{sec:exceptional}

The purpose of this section is extend $F$ to exceptional vertices $E$, so that we get  uniformly negative contribution from each vertex. The idea is to first begin with the idealized setting in which every vertex of \(B_j\) lies exactly at the phase \(\pi j/2\), while a single exceptional vertex \(u\) remains away from each cluster center. In \cref{lem:exceptional-ideal-case}, we show that in this case, one can choose a unit test value \(z_u\) for which the contribution of \(u\) to \eqref{eq:second-order-condition-statement} is at most \(-\kappa n\).

For a general clustered configuration, the cluster phases are only approximately centered, and the normalized edge weights satisfy only an approximate balance condition. However, we can use perturbative argument to see that exceptional vertices still supply sufficiently negative contribution needed to control the full quadratic form \(\cQ_{V\times V}\).

We first isolate the elementary trigonometric estimate used in the exceptional-vertex argument.

\begin{lemma}\label{lem:trig-ineq}
For $0<\phi<\pi/2$, let $\Xi(\phi):=\sqrt{1+\sin^2\phi}+\sqrt{1+\cos^2\phi}.$ Then
\[
    \frac{\Xi(\phi)}
    {\Xi(\phi)+\cos\phi+\sin\phi}<\frac1{\sqrt2}.
\]
\end{lemma}

\begin{proof}
Put $s=\sin\phi$, $c=\cos\phi$, and $p=sc>0$.  Then
\[
    \Xi(\phi)^2=3+2\sqrt{2+p^2},
    \qquad (s+c)^2=1+2p.
\]
The difference $ (3+2\sqrt2)(1+2p)-3-2\sqrt{2+p^2}$ vanishes at $p=0$ and has positive derivative for at positive $p$.  Since both
sides are positive, it follows that
\[
    \Xi(\phi)<(1+\sqrt2)(s+c).
\]
Since $x\mapsto x/(x+s+c)$ is increasing for $x>0$, the desired ratio is
strictly smaller than
\[
    \frac{1+\sqrt2}{2+\sqrt2}=\frac1{\sqrt2}. \qedhere
\]
\end{proof}

In the next lemma, we consider the case in which the vertices form four ideal clusters, with a unique exceptional vertex lying outside all four. We show that, if the first-order condition holds, one can assign a value to the exceptional vertex so that its contribution to the second-order condition is very negative. For simplicity, we adopt a different normalization in this lemma.

\begin{lemma}\label{lem:exceptional-ideal-case}
For all $q_0 > 1/\sqrt2$ and $h > 0$, there is a constant $\kappa>0$, such that the following  holds.  Suppose $\phi\in\T$ is at least $h/2$ from every multiple of $\pi/2$, and $W_0,W_1,W_2,W_3\in[0,1/4]$ satisfy
\[
    \sum_{j=0}^3W_j\ge q_0,
    \qquad
    \sum_{j=0}^3W_j\sin\paren{\frac{\pi j}{2}-\phi}=0.
\]
Then there is a $z\in\CC$ with $\abs z=1$ such that
\[
    \sum_{j=0}^3W_j\cos\paren{\frac{\pi j}{2}-\phi}
       \abs{z-i^j}^2\le-\kappa.
\]
\end{lemma}

\begin{proof}
Without loss of generality, we may assume $0<\phi<\pi/2$. The balance condition says
\begin{equation}\label{eq:complex-first-order-condition}
    \sum_{j=0}^3W_ji^j
    =(W_0-W_2)+i(W_1-W_3)=re^{i\phi}
\end{equation}
for some $r\in\RR$.  We set $z = H/|H|$, where 
\[
    H=(W_0+W_2)\cos\phi+i(W_1+W_3)\sin\phi.
\]

Since $W_0+W_2\ge q_0-1/2>0$, we have $H\ne0$. We first claim that $\abs H>r$.  This is immediate if $r\le0$.  If $r>0$ and
$\abs H\le r$, then \eqref{eq:complex-first-order-condition} gives
\[
  W_0-W_2=r\cos\phi,\qquad W_1-W_3=r\sin\phi.
\]
Note that this means $W_0+W_2\ge r\cos\phi$, and  $W_1+W_3\ge r\sin\phi$. Using $0\le W_j\le1/4$, we also get
\[
  W_0+W_2\le\frac12-r\cos\phi,
  \qquad
  W_1+W_3\le\frac12-r\sin\phi,
\]
while $|H| \leq r$ means
\[
\abs H^2
=(W_0+W_2)^2\cos^2\phi+(W_1+W_3)^2\sin^2\phi\le r^2.
\]
These inequalities imply that 
\begin{equation*}
    r^2 \geq r^2 \cos^4 \phi + |W_1 + W_3|^2 \sin^2 \phi,
\end{equation*}
which means $W_1+W_3\le r\sqrt{1+\cos^2\phi}$ and similarly $W_0+W_2\le r\sqrt{1+\sin^2\phi}$. Consequently, 
\[
\sum_jW_j
\le\min\bigl\{r\Xi(\phi),
       1-r(\cos\phi+\sin\phi)\bigr\}
\le\frac{\Xi(\phi)}
        {\Xi(\phi)+\cos\phi+\sin\phi}
<\frac1{\sqrt2}<q_0
\]
by \cref{lem:trig-ineq}, contradicting $\sum_jW_j\ge q_0$. Now a direct expansion shows
\begin{equation}\label{eq:ideal-test-expansion}
\sum_jW_j\cos\paren{\frac{\pi j}{2}-\phi}\abs{z-i^j}^2 =2r-2\Re(\overline zH) =2(r-\abs H)<0.
\end{equation}

Finally, the admissible weights and phases form a compact set, which therefore gives allows a uniform upper bound $-\kappa<0$.
\end{proof}

Finally, we can assign a value $z$ to the test function at each exceptional vertex while ensuring that $\mathcal{Q}_{V \times E}$ stays sufficiently negative. 

\begin{lemma}\label{lem:exceptional-vertex-main}
There is a constant $\tau_0>0$, depending only on $h$ and $q_0$, such that the
following holds.  Suppose $(G,\theta)$ is
$(h,\rho,\alpha,\beta)$-clustered with $\mu(G)\ge3/4-\eta$, and
\[
    0\le\eta,\rho,\alpha,\beta\le\tau_0.
\]
For every exceptional vertex $u\in E$, there is a $z_u\in\CC$ with
$\abs{z_u}=1$ such that
\begin{equation}\label{eq:exceptional-vertex-negative}
\sum_{j=0}^3\sum_{v\in B_j}
w_{uv}\cos(\theta_v-\theta_u)
\abs{z_u-i^j(1+it_v)}^2
\le-\frac{\kappa n}{2},
\end{equation}
where $\kappa$ is from \cref{lem:exceptional-ideal-case}.
\end{lemma}

\begin{proof}
Fix $u\in E$, put $\phi=\theta_u-\bar t$, and define
\[
    W_j:=\frac1n\sum_{v\in B_j}w_{uv}.
\]
Then $0\le W_j\le1/4+\rho$ and $W_0 + W_1 + W_2 + W_3 \geq 3/4 - \eta - \beta$. Since $\theta_v-\theta_u=\pi j/2-\phi+t_v$ for $v\in B_j$, the
first-order condition and the $L^1$ bound on the $t_v$ give
\begin{equation}\label{eq:perturbed-first-order}
\Bigl|{\sum_{j=0}^3W_j\sin\Bigl({\frac{\pi j}{2}-\phi}\Bigr)}\Bigr|
\le\alpha+\beta.
\end{equation}
We next perturb these weights to impose exact first-order balance.  For
$X\in\RR^4$, write
\[
    D_\phi(X):=\sum_{j=0}^3X_j
        \sin\paren{\frac{\pi j}{2}-\phi}.
\]
Set $\lambda=(1+4\rho)^{-1}$ and $W'_j=\lambda W_j$.  Then
$W'_j\in[0,1/4]$,
\[
    \sum_j\abs{W_j-W'_j}\le4\rho,
    \qquad
    D_\phi(W')=\lambda D_\phi(W).
\]
Put $s=\sin\phi$, $c=\cos\phi$, $d=W'_2-W'_0$, and
$e=W'_1-W'_3$, so that $D_\phi(W')=sd+ce$.  If
$\abs s\ge\abs c$, replace $(W'_0,W'_2)$ by a pair
$(\tilde W_0,\tilde W_2)\in[0,1/4]^2$ whose difference is
$-ce/s$, leaving the other two coordinates unchanged.  Such a pair exists
because $\abs{ce/s}\le1/4$, and it may be chosen at $\ell^1$ cost
\[
    \abs{(-ce/s)-d}
    =\frac{\abs{D_\phi(W')}}{\abs s}
    \le\sqrt2|{D_\phi(W')}|.
\]
If $\abs c>\abs s$, the symmetric construction instead adjusts
$(W'_1,W'_3)$.  In either case, we have
\[
    \tilde W_j\in[0,1/4],
    \qquad
    D_\phi(\tilde W)=0,
    \qquad
    \sum_j|{W_j-\tilde W_j}|
    \le K(\rho+\alpha+\beta)
\]
for an absolute constant $K$, where we used
\eqref{eq:perturbed-first-order} in the last inequality. For sufficiently small $\tau_0$, the perturbation estimate above imply
\[
    \sum_j\tilde W_j
    \ge\frac34-\eta-\beta-K(\rho+\alpha+\beta)\ge q_0.
\]
Moreover, because $u\in E$ and
$\abs{\bar t}\le h/2$ for small enough $\alpha$, the phase $\phi$ is at least $h/2$ from every
multiple of $\pi/2$.  We may therefore apply
\cref{lem:exceptional-ideal-case} and obtain $z_u$, with
$\abs{z_u}=1$, such that
\begin{equation}\label{eq:tilde-w-upper-bound}
\sum_j\tilde W_j\cos\paren{\frac{\pi j}{2}-\phi}
\abs{z_u-i^j}^2\le-\kappa.
\end{equation}

By the choice of $h_0$, we have $|{t_v}|\le h+\alpha<1$.  Because
$|{z_u-i^j}|\le2$, the difference between
$|{z_u-i^j(1+it_v)}|^2$ and $|{z_u-i^j}|^2$ is bounded by an absolute
constant times $\abs{t_v}$.
The cosine function is Lipschitz, and $|{(\theta_v-\theta_u)-(\pi j/2-\phi)}|=|{t_v}|.$ Consequently, using \eqref{eq:tilde-w-upper-bound},
\begin{equation*}
\frac1n\sum_j\sum_{v\in B_j}
w_{uv}\cos(\theta_v-\theta_u)
\abs{z_u-i^j(1+it_v)}^2\le-\kappa+C_2(\rho+\alpha+\beta)
\end{equation*}
for an absolute constant $C_2$.
After decreasing $\tau_0$, the right-hand side is at most
$-\kappa/2$.
\end{proof}

Putting all the pieces together, we obtain that $\mathcal{Q}_{V \times V}(F) \leq 0$ for a test function $F$; then $\mathcal{Q}_{V \times V}(F) = 0$ from the second-order condition, and this equality strongly constrains the graph and equilibrium structure.

\begin{proposition}\label{prop:eliminate-exceptional-vtx-and-perturbation}
There is a constant $\tau>0$ such that the following holds.  Suppose
$(G,\theta)$ is $(h,\rho,\alpha,\beta)$-clustered with
$\mu(G)\ge3/4-\eta$, and
\[
    0\le\eta,\rho,\alpha,\beta\le\tau.
\]
Then $E$ is empty and $t_v=0$ for every $v\in B$.
\end{proposition}

\begin{proof}
Without loss of generality, suppose $\beta$ is chosen as small as possible for $G$, meaning that $|E| = \beta n$. For $u\in B$, set $F(u)=F\vert_B(u)$ as before.  For each $u\in E$, set
$F(u)=z_u$, where $z_u$ is supplied by
\cref{lem:exceptional-vertex-main}.  Since $\abs{z_u}=1$,
\[
    \cQ_{E\times E}(F)\le4\abs E^2=4\beta^2n^2.
\]
Summing \eqref{eq:exceptional-vertex-negative} over $u\in E$ and using
\cref{cor:bulk-bound-main}, we find
\begin{align}\label{eq:full-test-upper-bound}
\cQ_{V\times V}(F)
&=\cQ_{B\times B}(F)+2\cQ_{B\times E}(F)
  +\cQ_{E\times E}(F)\notag\\
&\le\paren{-c+C(\rho+\eta)}n^2\sigma^2
  +n^2\beta(C\alpha-\kappa+4\beta).
\end{align}
Choose $\tau\le\tau_0$ sufficiently small, in terms of
$c,C$, and $\kappa$, that both coefficients in the last line are
negative whenever $0\le\eta,\rho,\alpha,\beta\le\tau$.
The second-order condition gives $\cQ_{V\times V}(F)\ge0$, whereas
\eqref{eq:full-test-upper-bound} would be strictly negative if either
$\sigma>0$ or $\beta>0$.  Hence $\sigma=\beta=0$, as required.
\end{proof}

\section{Proof of the main theorem}
\label{sec:exact-four-clusters}

The preceding argument in \cref{prop:eliminate-exceptional-vtx-and-perturbation} leaves only the exact quarter turn configuration to be dealt with. The following argument shows that such configuration can not be local minimum.

\begin{lemma}\label{lem:exact-four-cluster-final-step}
Suppose $\theta$ is a local minimum and
$V=B_0\sqcup B_1\sqcup B_2\sqcup B_3$, with
$\theta_v=\pi j/2$ for every $v\in B_j$.  Then every edge of $G$ has both
endpoints in the same set $B_j$.  Consequently,
\[
    \min_{v \in V} \deg (v)\le\max_{0\le j\le3}\paren{\abs{B_j}-1}.
\]
\end{lemma}

\begin{proof}
Set $F(v)=i^j$ for $v\in B_j$.  Same cluster edges make no contribution
to $\cQ_{V\times V}(F)$, adjacent-cluster edges have cosine zero, and
opposite-cluster edges have cosine $-1$.  Hence
\[
    0\le\cQ_{V\times V}(F)
    =-4\sum_{\mathrm{opp}}w_{uv},
\]
so there are no edges between opposite clusters. For $u\in B_j$, the equilibrium equation now reads
\[
    \deg_{B_{j-1}}(u)=\deg_{B_{j+1}}(u).
\]
Let $m_j$ be the number of edges between $B_j$ and $B_{j+1}$.  Summing
over $u\in B_j$ gives $m_{j-1}=m_j$.  Thus all four adjacent edge counts
have a common value $m$.

Suppose $m>0$.  Rotate $B_0$ by $t$ and $B_1$ by $-t$, leaving
$B_2\cup B_3$ fixed.  Since opposite edges are absent, an exact
calculation gives
\[
    \mathcal E_G(\theta^t)-\mathcal E_G(\theta)
    =m\paren{2\sin t-\sin(2t)}
    =2m\sin t\paren{1-\cos t}.
\]
This is negative for all sufficiently small negative $t$, contradicting local
minimality.  Hence $m=0$, so no edge joins distinct clusters.  The degree
bound follows immediately.
\end{proof}

We are now finally ready to prove the main theorem.
\begin{proof}[Proof of \cref{thm:main}]
Retain the fixed choice of $h\in(0,h_0)$, and let $\tau>0$ be supplied by
\cref{prop:eliminate-exceptional-vtx-and-perturbation}.  Choose $ 0<\rho=\alpha=\beta<\min\{\tau,1/4\}.$ Apply \cref{thm:clustering-main} with these parameters, and let
$\eta_{\mathrm{cl}}>0$ be the resulting constant.  Finally, choose
$0<\eta\le\min\{\eta_{\mathrm{cl}},\tau\}$ sufficiently small that
\begin{equation}\label{eq:final-degree-separation}
    \frac34-\eta>\frac14+\rho.
\end{equation}
Suppose, for a contradiction, that $G$ satisfies
\[
    \min_{v\in V}\deg(v)\ge\paren{\frac34-\eta}n
\]
and has a nonsynchronized local minimum $\theta$.  Then
\[
    \mu(G)=\frac{\min_{v\in V}\deg(v)+1}{n}
    \ge\frac34-\eta\ge\frac34-\eta_{\mathrm{cl}},
\]
so \cref{thm:clustering-main} shows that, after a global rotation,
$(G,\theta)$ is $(h,\rho,\alpha,\beta)$-clustered.  Global rotations
preserve local minimality, and $\eta,\rho,\alpha,\beta\le\tau$.  Hence
\cref{prop:eliminate-exceptional-vtx-and-perturbation} gives
$E=\varnothing$ and $t_v=0$ for every $v\in B=V$.  Thus
$x_v=\bar t$ for every $v\in B$.  After rotating all phases by
$-\bar t$, we therefore have $\theta_v={\pi j}/{2}$ for all $v \in B_j.$ Consequently,
\cref{lem:exact-four-cluster-final-step} gives
\[
    \min_{v \in V} \deg(v)\le\max_j\paren{\abs{B_j}-1}
    <\paren{\frac14+\rho}n,
\]
which contradicts the assumed minimum-degree bound and
\eqref{eq:final-degree-separation}.  Hence $G$ has no nonsynchronized
local minimum and is therefore globally synchronizing.
\end{proof}
\begin{remark}[Definition of global synchronization]
    \label{rem:final-remark}
    Throughout this paper, we adopt the definition of global synchronization from \cite[Conjecture~5]{randomstrasse-25}, which requires that $\mathcal{E}_G$ have no spurious local minima. Other authors, such as \cite{bandeira-expanders}, use a dynamical definition: the negative gradient flow of $\mathcal{E}_G$ must synchronize from almost every initial condition. This notion is slightly stronger than the one established here; see \cite[Definition~1.1 and the subsequent discussion]{bandeira-expanders}.

    Our result can be extended to this dynamical setting. It suffices to address the exactly clustered case, since in all other cases we show that every second order critical point of $\mathcal{E}_G$ is synchronized. Consequently, every nonsynchronized critical point is a strict saddle, and only a measure zero set of initial conditions can converge to it under the negative gradient flow \cite{bandeira-expanders}. To keep the paper concise and focused, we leave the extension to the exactly clustered case for future work.
\end{remark}

\bibliographystyle{plain}
\bibliography{references}

\end{document}